\documentclass[final,3p,times]{elsarticle}

\usepackage{amsmath}

\usepackage{amssymb}
\usepackage{bm}
\usepackage{amsthm}

\newtheorem{corollary}{Corollary}
\newtheorem{proposition}{Proposition}

\usepackage[colorlinks]{hyperref}
\usepackage{mathtools}
\mathtoolsset{showonlyrefs=true}
\usepackage{tikz}

\journal{Applied Mathematics Letters}

\newcommand{\dif}{\,\mathrm{d}}

\begin{document}

\begin{frontmatter}



\title{Error control technique of quadrature-based algorithms for the action of real powers of a Hermitian positive-definite matrix}


\author[nu]{Motohiro Otsuka\corref{cor1}}
\author[IHI]{Fuminori Tatsuoka}
\author[nu]{Tomohiro Sogabe}
\author[nu]{Kota Takeda}
\author[nu]{Shao-Liang Zhang}

\affiliation[nu]{organization={Department of Applied Physics, Graduate School of Engineering, Nagoya University},
            addressline={Furo-cho, Chikusa-ku},
            city={Nagoya},
            state={Aichi},
            postcode={464-8603},
            country={Japan}}
\affiliation[IHI]{organization={Corporate Research and Development Division, IHI Corporation},
            addressline={1 Shin-nakahara-cho},
            city={Isogo-ku},
            state={Yokohama},
            postcode={235-8501},
            country={Japan}}
\cortext[cor1]{Corresponding author}

\begin{abstract}
This study considers quadrature-based algorithms to compute $A^\alpha \bm{b}$, the action of a real power of a Hermitian positive-definite matrix $A$ on a vector $\bm{b}$.
In these algorithms, the computation of an integral representation of $A^{\alpha}\bm{b}$ is reduced to solving several tens or hundreds of shifted linear systems.
Current approaches usually analyze the quadrature discretization error, but rarely take into account the additional error introduced by solving these shifted linear systems with iterative solvers.
Here, we bound this error with the residual of the approximated solution of these linear systems.
This allows the derivation of a stopping criterion for iterative solvers to keep the error of $A^\alpha \bm{b}$ below a prescribed error tolerance.
Numerical results demonstrate that the proposed criterion enables the computation of $A^\alpha \bm{b}$ within prescribed tolerance limits.
\end{abstract}



\begin{keyword}
matrix function, matrix fractional power, numerical quadrature, error control


\end{keyword}

\end{frontmatter}


\section{Introduction}
\label{sec:introduction}
We consider the action of the matrix real power on a vector $\bm{b} \in \mathbb{C}^n$, given by $A^\alpha \bm{b}$ where the matrix $A \in  \mathbb{C}^{n\times n}$, is raised to the power of $\alpha$, a real number.
The matrix real power is defined via the matrix exponential $\exp(X) := I + X + X^2 / 2! + X^3 / 3! + \cdots ~ (X \in \mathbb{C}^{n \times n})$ and the matrix principal logarithm.
The matrix logarithm is any matrix $Y \in \mathbb{C}^{n \times n}$ that satisfies $\exp(Y) = A$.
If all the eigenvalues of $Y$ are contained in $\{ z \in \mathbb{C} : |\Im(z)| < \pi \}$, then $Y$ is called the principal logarithm of $A$ and is denoted $\log(A)$.
The principal matrix real power is then defined as $A^\alpha = \exp(\alpha \log(A))$, \cite{Higham2}.
When $0 < \alpha < 1$, $A^\alpha$ admits the following integral representation \cite[Eq.~(1.4)]{Higham2}:
\begin{equation}
A^\alpha = \frac{\sin(\alpha\pi)}{\alpha\pi} A \int_0^\infty (t^{1/\alpha} I + A)^{-1} \dif t \quad (0 < \alpha < 1).
\label{eq:integral_representation}
\end{equation}
Since $A^\alpha = A^{\lfloor \alpha \rfloor} A^{\alpha - \lfloor \alpha \rfloor}$ holds for any real $\alpha$, we focus on the cases where $0 < \alpha < 1$ entails no loss of generality.
Applications of the form of $A^\alpha \bm{b}$ arise in various scientific computing situations, e.g., numerical methods used for fractional differential equations \cite{ref_diff_eq1, ref_diff_eq2} and the computation of weighted matrix geometric means \cite{Fasi}.

Although several computational methods for computing $A^\alpha \bm{b}$ exist, we here focus on quadrature-based algorithms.
In these algorithms, Equation \eqref{eq:integral_representation} is discretized using a suitable change of variables and quadrature formulas, such as the Gauss--Jacobi (GJ) quadrature \cite{Cardoso, Fasi} and the double-exponential (DE) formula \cite{Tatsuoka1}.
This gives the approximation
\begin{equation}
  A^\alpha \bm{b} \approx \frac{\sin(\alpha\pi)}{\alpha\pi} \sum_{k = 1}^{m} w_k A (t_k^{1/\alpha}I + A)^{-1} \bm{b},
  \label{eq:quadrature_approx}
\end{equation}
where $t_k \geq 0$ are quadrature nodes and $w_k$ are the corresponding weights.
An advantage of quadrature-based algorithms is that, for large sparse matrices, they enable $A^\alpha \bm{b}$ to be computed without explicitly forming $A^\alpha$.
Moreover, the independence of the computations at the different quadrature nodes makes these methods well suited to parallel implementation.

Existing research into quadrature-based algorithms \cite{Cardoso, Fasi, Tatsuoka1} typically focuses on the quadrature discretization error, with the assumption that the integrand can be computed exactly.
However, when using an iterative method such as the shifted conjugate gradient (Shifted CG) method \cite{Krylov} for computing $(t^{1/\alpha}I+A)^{-1}\bm{b}$, the integrand is not computed exactly.
The resulting error must therefore be considered.
An unanswered question is how the stopping criterion in iterative methods should be set at each quadrature node.

In this paper, we analyze the integrand evaluation error for the case where $A$ is Hermitian positive definite.
We propose a stopping criterion for an iterative method for solving linear systems in the integrand at each quadrature node.
First, we derive an upper bound on the integrand evaluation error using the residuals of  approximate solutions to the linear systems $(t_k^{1/\alpha}I+A)^{-1}\bm{b}$.
This allows the integral evaluation error to be estimated both theoretically and numerically.
We hence propose the stopping criterion for an iterative method for each linear system to compute $A^\alpha \bm{b}$ with the prescribed accuracy.

The remainder of this paper is organized as follows.
Section~\ref{sec:error_analysis} analyzes the error arising from the inexact evaluation of the integrand.
A practical stopping criterion for the iterative solver is devised such that the overall approximation meets a prescribed tolerance condition.
Section~\ref{sec:numerical_experiments} reports numerical experiments that illustrate the effectiveness of the proposed error-control technique.
Finally, Section~\ref{sec:conclusion} finishes with concluding remarks.

\section{Error Analysis and Control}
\label{sec:error_analysis}

In this section, we analyze the integrand evaluation error and propose a stopping criterion of an iterative method for solving $(t_k^{1/\alpha}I+A)^{-1} \bm{b}$ that guarantees a prescribed overall error condition.

Two sources of error arise when computing $A^\alpha \bm{b}$ via Equation \eqref{eq:quadrature_approx}: the discretization error due to numerical quadrature and the error in evaluating the integrand at each quadrature node.
Taking the exact value of the integrand contribution at the $k$th node as
\begin{equation}
  \bm{g}_k := \frac{\sin (\alpha \pi)}{\alpha \pi} w_k A \left( t_k^{1/\alpha}I + A \right)^{-1} \bm{b},
\end{equation}
and its computed approximation denoted $\tilde{\bm{g}}_k$, the total error can be decomposed as
\begin{equation}
  \left\|
  A^\alpha \bm{b} - \sum_{k = 1}^{m} \tilde{\bm{g}}_k
  \right\|_2
  \leq
  \underbrace{
  \left\|
  A^\alpha \bm{b} - \sum_{k = 1}^{m} \bm{g}_k
  \right\|_2
  }_{\text{quadrature error}}
  +
  \underbrace{
  \left\|
  \sum_{k = 1}^{m} \bm{g}_k - \sum_{k = 1}^{m} \tilde{\bm{g}}_k
  \right\|_2
  }_{\text{integrand evaluation error}}
  \label{eq:error_decomposition}
\end{equation}
where $\|\cdot\|_2$ denotes the vector $2$-norm (Euclidean norm).

To bound the total error by a tolerance level $\epsilon$, we control the first term (quadrature error) and the second term (integrand evaluation error) on the right-hand side of Equation \eqref{eq:error_decomposition} so that neither of these quantities exceeds $\epsilon/2$.
The quadrature error was previously analyzed using the GJ quadrature \cite{Cardoso, Fasi} and the DE formula \cite{Tatsuoka1, Tatsuoka2}.
Therefore, we focus on controlling the integrand evaluation error so that it is bounded by $\epsilon/2$.

\subsection{Estimating the integrand evaluation error}

To ensure that the total integrand evaluation error is no greater than $\epsilon/2$, it suffices, by the triangle inequality, to bound the error at each quadrature node by $\epsilon/(2m)$:
\begin{align}
  \left\|
  \sum_{k = 1}^{m} \bm{g}_k - \sum_{k = 1}^{m} \tilde{\bm{g}}_k
  \right\|_2
  &\leq
  \sum_{k=1}^{m} \left\| \bm{g}_k - \tilde{\bm{g}}_k \right\|_2
  \leq
  \frac{\epsilon}{2}.
\end{align}
It therefore suffices to enforce, for each $k = 1, \ldots, m$, the condition
\begin{equation}
  \left\| \bm{g}_k - \tilde{\bm{g}}_k \right\|_2
  \leq
  \frac{\epsilon}{2m}.
  \label{eq:pointwise_error_bound}
\end{equation}

We present the following proposition to estimate the error at each quadrature node.
The error in computing $A(\sigma I + A)^{-1}\bm{b}$ is estimated using an approximate solution $\tilde{\bm{x}}$ of the shifted linear system $(\sigma I + A)\bm{x}=\bm{b}$ expanded in terms of the residual norm $\|\bm{b}-(\sigma I + A)\tilde{\bm{x}}\|_2$.

\begin{proposition}\label{prop:error_bound}
Let $A \in \mathbb{C}^{n \times n}$ be Hermitian positive definite, and let $\lambda_{\text{max}}$ denote its largest eigenvalue.
Then, for any real $\sigma \geq 0$ and any $\bm{x} \in \mathbb{C}^{n}$, the following bound holds:
 \begin{equation}
   \left\|
   A \left( \sigma I + A \right)^{-1} \bm{b} - A \bm{x}
   \right\|_2
   \leq
   \frac{1}{1 + \sigma / \lambda_{\text{max}}}
   \left\|
   \bm{b} - \left( \sigma I + A \right)\bm{x}
   \right\|_2.
   \label{eq:prop_bound}
 \end{equation}
\end{proposition}

\begin{proof}
By the definition of the induced operator norm, we have
\begin{align}
  \left\|
  A \left( \sigma I + A \right)^{-1} \bm{b} - A \bm{x}
  \right\|_2
  &=
  \left\|
  A \left( \sigma I + A \right)^{-1} \left( \bm{b} - \left( \sigma I + A \right) \bm{x} \right)
  \right\|_2 \notag \\
  & \leq
  \left\|
   A \left( \sigma I + A \right)^{-1}
  \right\|_2
  \left\|
  \bm{b} - \left( \sigma I + A \right) \bm{x}
  \right\|_2.
  \label{eq:triangle_ineq}
\end{align}
For $\lambda>0$, the scalar function $f(\lambda)=\lambda/(\lambda+\sigma)$ is positive and increasing monotonically.
Since $A$ is Hermitian positive definite, it follows that
\begin{equation}
  \left\|
   A \left( \sigma I + A \right)^{-1}
  \right\|_2
  =
  \max_{\lambda \in \sigma(A)} \frac{\lambda}{\lambda + \sigma}
  =
  \frac{\lambda_{\text{max}}}{\lambda_{\text{max}} + \sigma}
  =
  \frac{1}{1 + \sigma / \lambda_{\text{max}}},
  \label{eq:spectral_norm}
\end{equation}
where $\sigma(A)$ denotes the spectrum (the set of eigenvalues) of $A$.
Combining Equation \eqref{eq:triangle_ineq} and \eqref{eq:spectral_norm} gives Equation \eqref{eq:prop_bound}.
\end{proof}

The bound in Proposition~\ref{prop:error_bound} implies that the factor $1/(1+\sigma/\lambda_{\text{max}})$ decreases as the shift parameter $\sigma$ increases, leading to a tighter estimate in terms of the residual norm.
Conversely, for small $\sigma$, this factor approaches $1$; the residual norm essentially acts directly as an upper bound for the error.
This behavior is illustrated in Figure \ref{fig:sigma_ex} with a concrete example involving a matrix $A$ that is the two-dimensional Laplacian discretized by the five-point centered difference scheme.
As shown, the error is tightly bounded according to the proposed upper bound.

\begin{figure}[htbp]
  \centering
  \includegraphics[width=\textwidth]{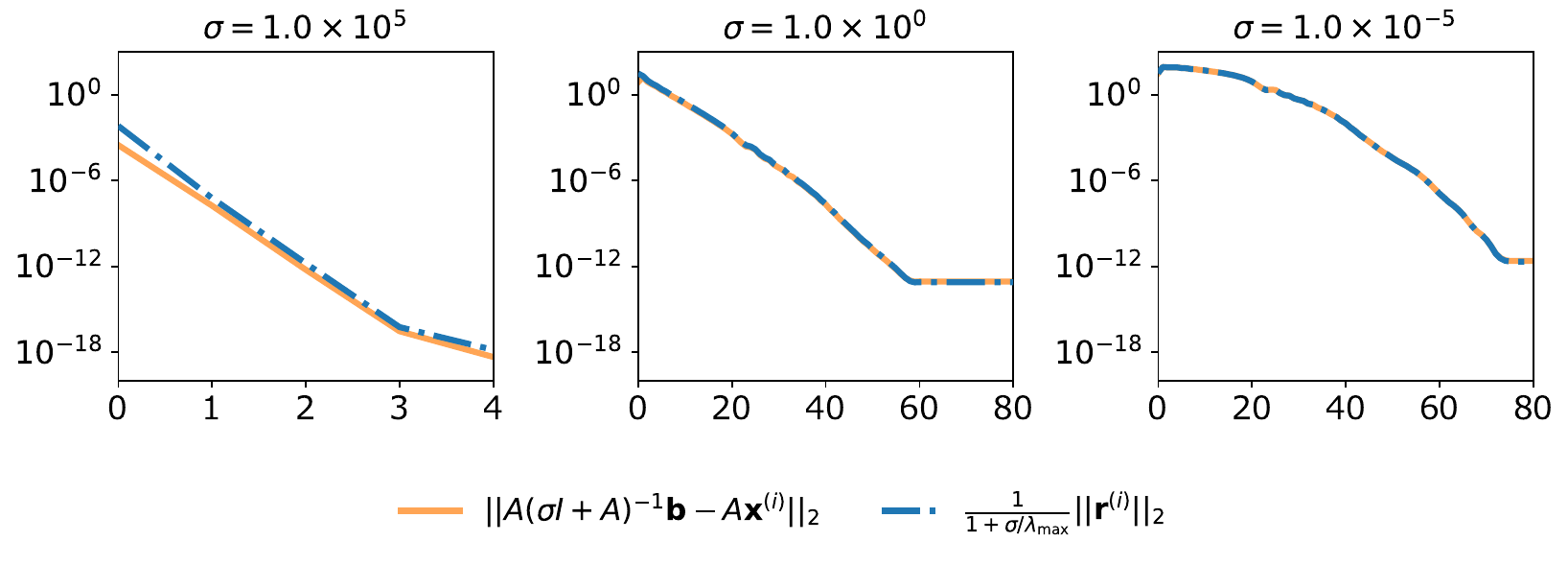}
  \caption{
  Illustrative example of the upper bound formulated by Proposition~\ref{prop:error_bound}.
  The vertical axis represents the absolute error in the $2$-norm, and the horizontal axis the number of iterations $i$ of the Shifted CG method.
  The coefficient matrix $A \in \mathbb{R}^{1024 \times 1024}$ is the finite-difference matrix of the two-dimensional Laplacian.
  The right-hand-side vector in the linear systems is $\bm{b} = [1,\dots,1]^\top$.
  }
  \label{fig:sigma_ex}
\end{figure}

\subsection{Derivation of a stopping criterion}

As a corollary, we use Proposition~\ref{prop:error_bound} to derive a stopping criterion for the iterative solution at each quadrature node.

\begin{corollary}\label{cor:convergence_criterion}
Let $\alpha \in (0,1)$.
For the $k$th shifted linear system $\left( t_k^{1/\alpha}I + A \right) \bm{x}_k = \bm{b}$, let $\tilde{\bm{x}}_k$ be an approximate solution and define its residual as
$\bm{r}_k := \bm{b} - \left( t_k^{1/\alpha}I + A \right) \tilde{\bm{x}}_k$.
If, for all $k=1,\dots,m$,
 \begin{equation}
   \left\| \bm{r}_k \right\|_2 \leq \frac{\epsilon}{2m} \frac{\alpha \pi}{\sin (\alpha \pi)} \frac{1 + t_k^{1/\alpha} / \lambda_{\text{max}}}{w_k},
   \label{eq:convergence_criterion}
 \end{equation}
then, the integrand evaluation error is bounded by $\epsilon/2$.
\end{corollary}

\begin{proof}
In Proposition~\ref{prop:error_bound}, set $\sigma = t_k^{1/\alpha}$ and $\bm{x} = \tilde{\bm{x}}_k$.
Then,
\begin{equation}
  \left\| \bm{g}_k - \tilde{\bm{g}}_k \right\|_2
  =
  \left\|
  \frac{\sin (\alpha \pi)}{\alpha \pi} w_k A \left( t_k^{1/\alpha}I + A \right)^{-1} \bm{b}
  - \frac{\sin (\alpha \pi)}{\alpha \pi} w_k A \tilde{\bm{x}}_k
  \right\|_2
  \leq
  \frac{\sin (\alpha \pi)}{\alpha \pi} w_k \frac{1}{1 + t_k^{1/\alpha} / \lambda_{\text{max}}}
  \left\| \bm{r}_k \right\|_2.
\end{equation}
By Equation \eqref{eq:pointwise_error_bound}, it suffices to enforce $\left\| \bm{g}_k - \tilde{\bm{g}}_k \right\|_2 \le \epsilon/(2m)$. A rearrangement of the terms yields Equation \eqref{eq:convergence_criterion}.
\end{proof}

Corollary~\ref{cor:convergence_criterion} implies that, when iteratively solving the shifted systems at each quadrature node, iterations may be terminated once the residual norm $\|\bm{r}_k\|_2$ falls below the right-hand side of Equation \eqref{eq:convergence_criterion}.
This constitutes the error control technique proposed in this paper.

A concrete example of the stopping criterion applied to each shift parameter when using the DE formula is shown in Figure~\ref{fig:bound_DE}.
The stopping criterion is relaxed for both extremely large and extremely small shift parameters, and it becomes most stringent for intermediate shifts.
This behavior reflects the fact that, in Equation \eqref{eq:convergence_criterion}, both the magnitude of the shift parameter $t_k^{1/\alpha}$ and the quadrature weight $w_k$ influence the stopping criterion.

\begin{figure}[htbp]
  \centering
  \includegraphics[width=0.7\textwidth]{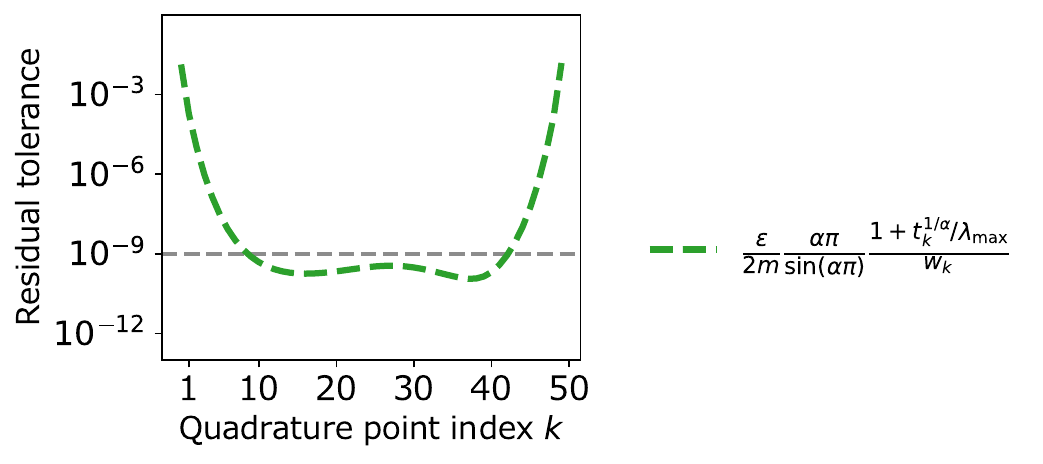}
  \caption{
  Stopping criterion for the shifted systems at the quadrature nodes derived from Corollary~\ref{cor:convergence_criterion} (DE formula, $\alpha = 0.2$, $\epsilon = 10^{-9}$).
  The vertical axis represents the admissible residual norm, and the horizontal axis the quadrature-node index $k$.
  }
  \label{fig:bound_DE}
\end{figure}

\section{Numerical Experiments}
\label{sec:numerical_experiments}

In this section, we verify that the proposed error-control technique enables the computation of $A^\alpha \bm{b}$ to a prescribed accuracy level.
All numerical experiments were performed using Python~3.13.1.
Computations were performed on an HP EliteBook 630 13-inch G9 Notebook PC equipped with a 12th Gen Intel(R) Core(TM) i5-1235U (1.30~GHz) CPU and 32~GB RAM.
Unless otherwise stated, IEEE double-precision floating-point arithmetic was used.

For the test matrices, we used finite-difference discretizations of the one-dimensional and two-dimensional Laplace operators, obtained from the three- and five-point centered difference stencils, respectively.
The corresponding matrix sizes were $n=1000$ and $n=1024$.
The right-hand side vector was chosen as $\bm{b} = [1,\dots,1]^\top$.

We computed $A^\alpha$ for $\alpha = 0.2$ and $0.5$, with total error tolerances $\epsilon = 10^{-3}, 10^{-6}$ and $10^{-9}$.
For the quadrature rules, we employed the GJ quadrature (GJ1 \cite[Eq.~(2.6)]{Cardoso} and GJ2 \cite[Eq.~(7)]{Fasi}) and the DE formula \cite[Alg.~1]{Tatsuoka1}.
The number of quadrature nodes $m$ was determined from scalar quadrature error bounds, as done elsewhere \cite{Tatsuoka1,Tatsuoka2}.
Specifically, for HPD matrices, the matrix error requirement was reduced to checking the scalar approximation error of $\lambda^\alpha$ at an extremal eigenvalue ($\lambda_{\max}$ or $\lambda_{\min}$, depending on the theorem setting).
We therefore chose the smallest $m$ such that the corresponding scalar error bound was below the quadrature error budget (here, $\epsilon/2$).
The shifted linear systems at each quadrature node were solved by using the Shifted CG method.
We used Equation \eqref{eq:convergence_criterion}, derived in Corollary~\ref{cor:convergence_criterion}, as the stopping criterion for each shifted system.
Reference solutions were computed using the \texttt{fractional\_matrix\_power} function in the SciPy library.

\begin{figure}[htbp]
  \centering
  \begin{tikzpicture}
    \node[inner sep=0pt, outer sep=0pt, anchor=north west] (panel) at (0,0) {%
      \begin{minipage}[c]{0.74\textwidth}
        \begin{minipage}{0.48\textwidth}
          \centering
          \includegraphics[width=0.9\textwidth]{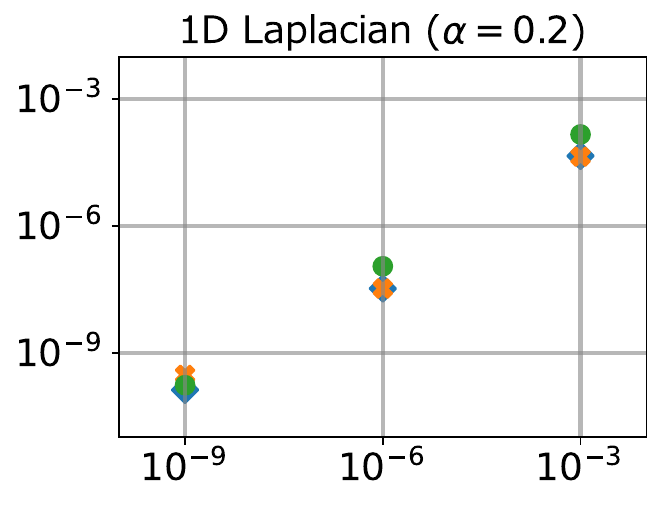}
        \end{minipage}%
        \hspace{0.005\textwidth}%
        \begin{minipage}{0.48\textwidth}
          \centering
          \includegraphics[width=0.9\textwidth]{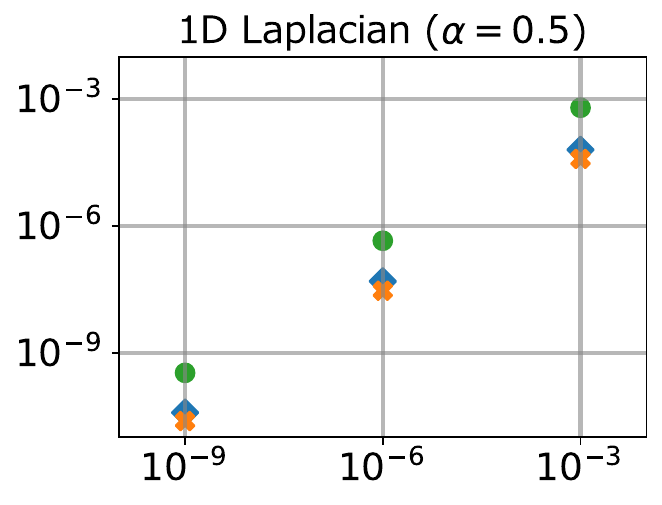}
        \end{minipage}

        \vspace{0.12cm}

        \begin{minipage}{0.48\textwidth}
          \centering
          \includegraphics[width=0.9\textwidth]{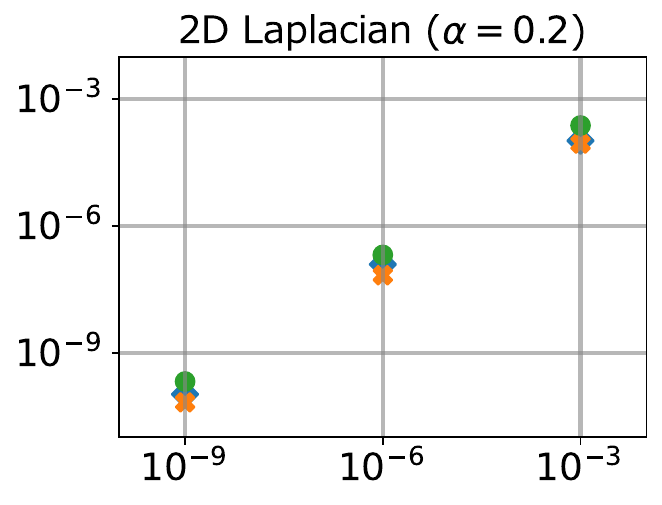}
        \end{minipage}%
        \hspace{0.005\textwidth}%
        \begin{minipage}{0.48\textwidth}
          \centering
          \includegraphics[width=0.9\textwidth]{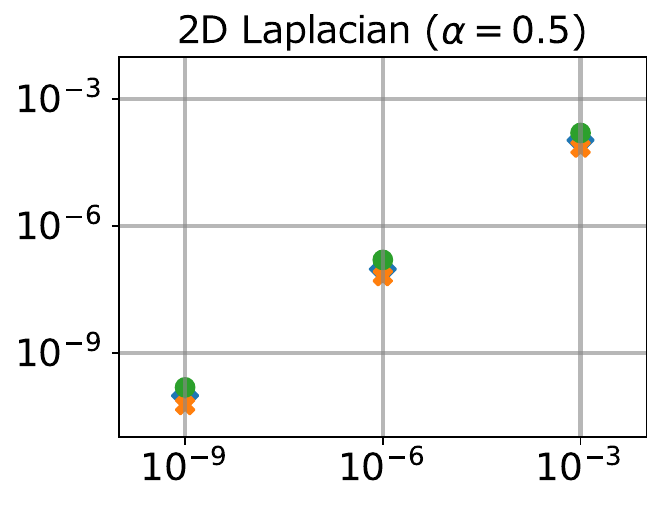}
        \end{minipage}

        \vspace{0.05cm}

        \begin{center}
          \includegraphics[width=0.3\textwidth]{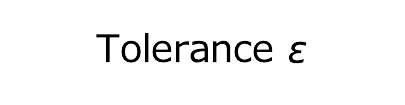}
        \end{center}
      \end{minipage}%
    };
    \node[anchor=center] at ([xshift=-9mm,yshift=9mm]panel.west)
      {\includegraphics[width=0.07\textwidth]{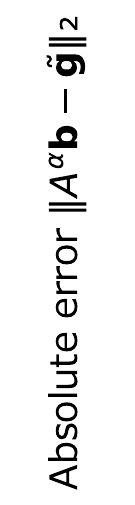}};
    \node[anchor=north west] at ([xshift=2mm,yshift=0mm]panel.north east)
      {\includegraphics[width=0.12\textwidth]{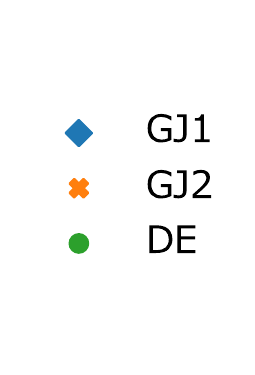}};
  \end{tikzpicture}

  \caption{
  Absolute errors of $A^\alpha \bm{b}$ computed using quadrature-based algorithms.
  Left column: $\alpha = 0.2$; right column: $\alpha = 0.5$.
  Top row: one-dimensional Laplacian; bottom row: two-dimensional Laplacian.
  The horizontal axis represents the prescribed error tolerance $\epsilon$, and the vertical axis the absolute error.
  }
  \label{fig:result}
\end{figure}

Absolute errors $\| A^\alpha \bm{b} - \sum_{k=1}^{m} \tilde{\bm{g}}_k \|_2$ computed for each test matrix and each value of $\alpha$ are reported in Figure~\ref{fig:result}.
For all the test matrices, all values of $\alpha$, and all quadrature rules (GJ1, GJ2, and DE), the absolute errors remain below the prescribed error tolerance $\epsilon$, thereby confirming the effectiveness of the per-node stopping criterion derived in Corollary~\ref{cor:convergence_criterion}.

\section{Conclusion}
\label{sec:conclusion}

In this study, we analyzed the integrand evaluation error of quadrature-based algorithms used for computing $A^\alpha \bm{b}$.
We proposed a stopping criterion for iterative methods used for solving shifted linear systems in the integrand to keep the total error below a prescribed error tolerance.

The present results open several avenues for future work.
First, we used a simple error allocation in which the quadrature discretization error and the integrand evaluation error were each controlled to be no greater than $\epsilon/2$; optimizing the allocation between these two error sources may therefore reduce the computational cost.
Second, while our analysis was restricted to the case where $A$ is Hermitian positive-definite, there is a need to extend the proposed error-control technique to more general matrices.
Third, investigating the applicability of the proposed approach to other matrix functions beyond matrix real powers would open new possibilities.
These functions would include the matrix sign function and functionals arising in applications, such as the von Neumann entropy.

\section*{Acknowledgments}
This work was supported by JSPS KAKENHI Grant Number 25H00449.



 \bibliographystyle{elsarticle-num}
 \bibliography{reference}





\end{document}